\documentclass[11pt]{article}
\usepackage[utf8]{inputenc}
\usepackage[colorlinks=true,urlcolor=blue,
citecolor=red,linkcolor=blue,linktocpage,pdfpagelabels,bookmarksnumbered,bookmarksopen]{hyperref}
\usepackage[numbers]{natbib}

\usepackage{amsmath}
\usepackage{amssymb}
\usepackage{latexsym}
\usepackage{enumitem}
\usepackage{mathrsfs}
\usepackage{comment}
\usepackage{color}
\usepackage{bbm}

\newtheorem{assumption}{Assumption}

\def\qed{ \ \vrule width.2cm height.2cm depth0cm\smallskip}
\newenvironment{proof}{\noindent {\bf Proof.\/}}{$\qed$\vskip 0.1in}

\newcommand{\ba}{\begin{array}}
\newcommand{\ea}{\end{array}}
\newcommand{\be}{\begin{equation}}
\newcommand{\ee}{\end{equation}}
\newcommand{\bea}{\begin{eqnarray}}
\newcommand{\eea}{\end{eqnarray}}
\newcommand{\beaa}{\begin{eqnarray*}}
\newcommand{\eeaa}{\end{eqnarray*}}

\def\dbE{\mathbb{E}}
\def\dbF{\mathbb{F}}

\def\dbL{\mathbb{L}}

\def\dbN{\mathbb{N}}
\def\dbP{\mathbb{P}}
\def\dbR{\mathbb{R}}

\def\dbV{\mathbb{V}}

\def\b{\beta}

\def\d{\delta}
\def\e{\varepsilon}

\def\si{\sigma}
\def\t{\tau}

\def\th{\theta}

\def\G{\Gamma}
\def\D{\Delta}
\def\Th{\Theta}

\def\cA{{\cal A}}

\def\cE{{\cal E}}
\def\cF{{\cal F}}

\def\cJ{{\cal J}}

\def\cL{{\cal L}}
\def\cM{{\cal M}}

\def\cP{{\cal P}}

\def\cX{{\cal X}}
\def\cY{{\cal Y}}
\def\cZ{{\cal Z}}

\def\no{\noindent}

\def\ms{\medskip}
\def\bs{\bigskip}
\def\q{\quad}
\def\qq{\qquad}

\def\pa{\partial}
\def\cd{\cdot}
\def\cds{\cdots}

\def\qed{ \hfill \vrule width.25cm height.25cm depth0cm\smallskip}

\newcommand{\basa}{\begin{assumption}}
\newcommand{\easa}{\end{assumption}}

\newcommand{\bas}{\begin{assum}}
\newcommand{\eas}{\end{assum}}

\def\pa{\partial}

 \def\cd{\cdot}
\def\cds{\cdots}

\def\dis{\displaystyle}

\def\1{{\bf 1}}

\def\:{\!:\!}
\def\reff#1{{\rm(\ref{#1})}}
\def \proof{{\noindent \bf Proof\quad}}

at 9pt

\newtheorem{thm}{Theorem}[section]

\newtheorem{prop}[thm]{Proposition}
\newtheorem{rem}[thm]{Remark}
\newtheorem{eg}[thm]{Example}
\newtheorem{defn}[thm]{Definition}
\newtheorem{assum}[thm]{Assumption}

\numberwithin{equation}{section}

\DeclareMathOperator*{\argmax}{arg\,max}

\begin{document}

\title{\bf A New Approach for the Continuous Time  Kyle-Back Strategic Insider Equilibrium Problem}
\author{, Jianfeng Zhang}

\author{
Bixing Qiao\thanks{\noindent USC Mathematics Department, Los Angeles, 90089; email: bqiao@usc.edu.}
 ~ and ~ Jianfeng Zhang \thanks{ \noindent USC Mathematics Department, Los Angeles, 90089;
email: jianfenz@usc.edu. This author is supported in part by NSF grant  \#DMS-2205972. }}

\date{\today}
\maketitle

\begin{abstract}
This paper considers a continuous time Kyle-Back model which is a game problem between an insider and a market marker. The existing literature typically focuses on the existence of equilibrium by using the PDE approach, which requires certain Markovian structure and the equilibrium is in the bridge form. We shall provide a new approach which is used widely for stochastic controls and stochastic differential games. We characterize all equilibria through a coupled system of forward backward SDEs, where the forward one is the conditional law of the inside information and the backward one is the insider's optimal value. In particular, when the time duration is small, we show that the FBSDE is wellposed and thus the game has a unique equilibrium. This is the first uniqueness result  in the literature, without restricting the equilibria to certain special structure. Moreover, this unique equilibrium may not be Markovian, indicating that the PDE approach cannot work in this case. We next study the set value of the game, which roughly speaking is the set of insider's values over all equilibria and thus is by nature unique. We show that, although the bridge type of equilibria in the literature does not satisfy the required integrability for our equilibria, its truncation serves as a desired approximate equilibrium and its value belongs to our set value. Finally, we characterize our set value through a level set of certain standard HJB equation. 
 \end{abstract}

 \bs
\no{\bf Keywords.} Kyle-Back model, insider trading, Nash equilibrium, forward backward SDE, set value

\bs

\no{\it 2020 AMS Mathematics subject classification:} 91A27, 91G15,  60H10

\vfill\eject
\section{Introduction} \label{sect-introduction}
The Kyle-Back model considers a financial market with three types of agents: 1) A representative {\it noise trader} who is non-strategic and provides liquidity; 2) an {\it insider} who observes the true value of the asset, which is the inside information, and chooses certain trading strategy to maximize his expected gain; 
 and 3)  a {\it market maker} who observes the total trading volume of the noise trader and the insider, but not the inside information, and chooses a pricing rule to determine the market price of the asset for break-even. This is a game  between the insider and the market marker. Such a model, as its name suggests, was initiated by Kyle  \cite{Kyle1985} (mainly) on discrete time models and Back \cite{Back1992} on continuous time models. The Kyle-Back model and its various variations have been extensively studied in  the mathematical finance literature. For continuous time models, we refer to Aase-Bjuland-Øksendal  \cite{Aase2012}, Back-Baruch \cite{Back2004}, Back-Cocquemas-Ekren-Lioui \cite{back2021}, Back-Pedersen \cite{BackPedersen}, Biagini-Hu-Meyer-Brandis-{\O}ksendal \cite{BHMO}, Bose-Ekren \cite{Bose2023, Bose2024}, Caldentey-Stacchetti  \cite{Caldentey2010}, Campi-\c{C}etin-Danilova \cite{CAMPI2011},  \c{C}etin \cite{Cetin2023, Cetin2023b},  \c{C}etin-Danilova \cite{Cetin2016, CD2}, \c{C}etin-Xing \cite{CetinXing2013}, Cho \cite{Cho2003}, Choi-Kwon-Kasper  \cite{Kasper2023}, Collin-Fos \cite{CollinFos2016}, Danilova \cite{Danilova2010}, Ekren-Mostowsk-Žitković \cite{ekren2022}, Lasserre \cite{Lasserre2004}, Ma-Sun-Zhou \cite{MSZ}, Ma-Tan \cite{MaTan2022}, and the references therein. In particular, we note that  \cite{back2021, Bose2023, Bose2024} proposed an optimal transport approach  for the problem and \cite{Cetin2023, Cetin2023b} studied the model with legal risk.

The existing literature mainly focuses on the existence of equilibrium by using the PDE approach. This approach requires certain Markovian structure of the price process under some reference measure. Moreover, the constructed equilibrium typically has the bridge structure, and is often not square integrable. While there have been many deep and highly technical results, the above features make it hard to build a general theory. In particular, the existence in a general model remains a challenging problem, and to the best of our knowledge, there are very few uniqueness results and all of them restrict the pricing rule to some special structure, e.g. Markovian structure in \cite{Back1992} and linear structure in \cite{Cho2003}. 

In this paper we propose to use the weak formulation through the Girsanov theorem to formulate the Kyle-Back model, with possibly a legal risk or cost. This is more or less a standard formulation  for stochastic controls and stochastic differential games, but seems new in the literature of Kyle-Back model. Thus we provide a unified approach for the Kyle-Back model, which is by nature a game, and the standard nonzero-sum games. This also enables us to use many ideas and technical tools from the stochastic differential game literature. Moreover, our approach does not require any more the Markovian structure of the price process, and thus helps for the existence of equilibrium. We shall note that, in order to apply the Girsanov theorem, we will require stronger integrability condition on the insider's admissible strategies, and for simplicity we will actually assume they are bounded. Other than this technical difference, our problem is equivalent to the standard models studied in the literature.  

Our first result is a full characterization of the equilibria in the Kyle-Back model thorough a coupled system of Forward Backward SDEs (FBSDEs, for short), under natural technical conditions. That is, the equilibria of the game have one to one correspondence with the solutions of the FBSDE. The forward component of our FBSDE is the conditional law of the inside information, conditional on the total trading volume which is observable to the market marker, and the backward component is the insider's optimal value. This mimics the structure of the mean field game systems, see e.g. \cite{CD1}, and it also provides intrinsic connection with the optimal transport approach proposed by \cite{back2021, Bose2023, Bose2024}. We shall explore these interesting connections in future research.   

When the time duration is small, we show that the FBSDE is wellposed, and thus the game has a unique equilibrium. To the best of our knowledge, this is the first uniqueness result for continuous time Kyle-Back model, without restricting the equilibria to certain special structure. Moreover, when the inside information is Bernoulli, namely taking only two values, the equilibrium price process is indeed Markovian under the appropriate reference measure. However, such Markov property often fails when the inside information has three or more possible values. Since the equilibrium here is unique, this implies that no equilibrium would satisfy the Makov property required in the standard literature and consequently this equilibrium is out of reach of the PDE approach. So our existence result is also new in this case, demonstrating further the power of our new approach.  
 
 We next study the set value of the game, which was initiated by Feistein-Rudloff-Zhang \cite{FRZ2021} for finitely many player games and Iseri-Zhang \cite{IseriZhang} for mean field games. Roughly speaking, the set value of a game is the set of values over all equilibria. However, we do not restrict to true equilibria, which in many cases could be a tall order, but consider asymptotic values over approximate equilibria. By considering the empty set as a possible set value, both the existence and the uniqueness of the set value are guaranteed by definition. We emphasize that, instead of focusing on constructing one equilibrium as in the literature, we investigate {\it all} equilibria and their values.  

We recall again that the bridge type of equilibria found in the literature typically do not satisfy the integrability/boundedness requirement, and thus cannot be viewed as a true equilibrium in our setting. However, by using Back 
\cite{Back1992}'s original model as an example, we show that a truncated version of this equilibrium serves as a desired approximate equilibrium in our sense, and the value of this equilibrium indeed belongs to the set value. So the new approach here still covers the classical results, but in terms of values instead of equilibria. 
 
 Finally, inspired by \cite{FRZ2021}, we provide a level set approach to characterize the set value. To be precise, we show that the set value is the zero level set of an auxiliary control problem.  This characterization can be viewed as a duality result, and could be potentially useful for numerical computation of the set value. 

The rest of the paper is organized as follows. In Section \ref{sect-model} we introduce the Kyle-Back model, both the standard formulation and our weak formulation. In Section \ref{sect-FBSDE} we characterize all equilibria with an FBSDE. In Section \ref{sect-smallT} we establish the wellposedness of the FBSDE when time duration is small. In Section \ref{sect-setvalue} we introduce the set value of the game and show through an example that the value of the bridge type of equilibria in the existing literature belongs to our set value. In Section \ref{sect-levelset} we establish the duality between the set value and the zero level set of the auxiliary control problem. Finally in Appendix we complete some proofs.

\section{The Kyle-Back model}
\label{sect-model}

\subsection{The standard formulation}
\label{sect-standardsetting}
Fix a filtered probability space $(\Omega, \dbF, \dbP)$. The Kyle-Back model concerns a financial asset on $[0, T]$ with fundamental value $V$ at $T$ and market price $S_t$, $0\le t\le T$.  It involves three types of agents:
\begin{itemize}
\item A {\it  representative noise trader}, who is non-strategic and places an accumulative order $B_t$ on $[0, t]$, which is modeled   as a  Brownian motion due to the central limit theorem. 

\item An {\it  inside trader}, or simply {\it insider}, who observes the value $V$ at initial time $0$ and the market price $S$ dynamically, and place an accumulative order $\Th_t = \Th(V; t, S_{[0, t]})$. 

\item {\it A market maker}, who observes the total accumulative order $Q_t := \Th_t + B_t$, but not the values $\Th, B$, or $V$, separately, and chooses a pricing rule $P$ to determine the market price: $S_t = P(t, Q_{[0, t]})$. The market marker has a prior distribution $\mu$ for $V$.
\end{itemize}

\no Notice that $\Th_t$ can be rewritten further through the pricing rule $P$ as a function of $(V; t, Q_{[0, t]})$. For technical reasons in this paper we shall assume $\Th$ takes this form directly:
\bea
\label{ThQ}
\Th_t = \Th(V; t, Q_{[0, t]}).
\eea
This can be justified by either assuming the market marker's pricing rule $P$ is one to one between the paths of $Q$ and $S$, as quite often the case in the literature, or by assuming that the insider also observes $Q$ (or equivalently $B$, since the insider knows $\Th$). Moreover, in this paper we shall assume $\Th$ is absolutely continuous in $t$:
\bea
\label{Thth}
\Th_t = \int_0^t \th(V; s, Q_{[0, s]})ds.
\eea
Then $Q = Q^{\th(V;\cd)}$, where for each $\th'=\th'(t, Q_{[0, t]})$, $Q^{\th'}$ is determined by $\th'$ as follows:
\bea
\label{Qth}
Q^{\th'}_t = \int_0^t \th'(s, Q^{\th'}_{[0,s]}) ds + B_t.
\eea
Here $\th'$ corresponds to $\th^v:=\th(v; \cd)$ for $V=v$, and we assume the above has a unique solution in appropriate sense. 

Given $V=v$ and a pricing rule $P$, the insider aims to maximize his expected gain:
\bea
\label{JPth0}
\sup_{\th'} J_0(P; v, \th'),\q\mbox{where}\q  J_0(P; v, \th') := \dbE\Big[\int_0^T [ v- P(t, Q^{\th'}_{[0, t]})] \th'(t,  Q^{\th'}_{[0, t]})dt\Big].
\eea
On the other hand, given $\th$, the market maker sets a rational pricing rule $P^\th$ for break-even:
\bea
\label{Pth0}
P^\th(t, Q^{\th(V;\cd)}_{[0, t]}) := \dbE\big[V |\cF^{Q^{\th(V;\cd)}}_t\big],\q 0\le t\le T,~\dbP\mbox{-a.s.}
\eea
This can be justified by a Bertrand competition argument, see e.g. \cite{Kyle1985}. 
\begin{defn}
\label{defn-NE0}
We say a pair $(P^*, \th^*)$ is an equilibrium of the model \reff{JPth0}-\reff{Pth0} if $P^* = P^{\th^*}$ as in \reff{Pth0}, and $\th^*(v; \cd) = \mbox{argmax}_{\th'} ~ J_0(P^*; v, \th')$ as in \reff{JPth0}, for all $v\in \mbox{Supp}(V)$.
\end{defn}

\begin{rem}
\label{rem-game}
(i) The market maker's problem \reff{Pth0} can be alternatively written as an optimization problem: 
\bea
\label{Pth2}
P^\th := \mbox{argmax}_P J_1(\th; P),\q \mbox{where}\q J_1(\th; P):= -\dbE\Big[\big|V-P(t, Q^{\th(V;\cd)}_{[0, t]}) \big|^2\Big].
\eea
Then $(P^*, \th^*)$ is the standard Nash equilibrium of the following nonzero sum game:
\bea
\label{game1}
J_0(P^*; v, \th^{*, v}) \ge J_0(P^*; v, \th'),~\forall \th', v;\q \mbox{and}\q J_1(\th^*; P^*) \ge J_1(\th^*; P),~ \forall P.
\eea

(ii) Assuming the market marker has a target strategy $\th$, namely he would use the pricing rule $P^\th$, the insider could choose a different strategy $\tilde \th$. Introduce:
\bea
\label{Jthth}
\cJ_0(\th, \tilde \th) := \int_\dbR J_0(P^\th; v, \tilde \th^v)\mu(dv),\q \cJ_1(\th, \tilde\th) := J_1(\tilde\th; \dbP^\th).
\eea
Then $\th^*$ is an equilibrium (with corresponding $P^* = P^{\th^*})$ if only if $(\th^*, \th^*)$ is the Nash equilibrium of the following standard nonzero sum game:
\bea
\label{Game2}
\cJ_0(\th^*, \th^*)\ge \cJ_0(\th^*,  \tilde \th),~ \forall \tilde \th;\q\mbox{and}\q \cJ_1(\th^*, \th^*)\ge \cJ_1(\th,  \th^*),~ \forall \th.
\eea

(iii) The game problem \reff{Game2} is fundamentally different from the control problem:
\bea
\label{control}
\sup_\th \cJ_0(\th, \th),
\eea
which is studied, for example, in \cite{MSZ, MaTan2022}. We remark that, since $P^\th = \argmax_P J_1(\th; P)$, we may view \reff{control} as a Stackelberg game, where the insider is the leader and the market marker is the follower.
\end{rem}

\begin{rem}
\label{rem-Kyleexample}
(i) Assume $T=1$ and $V\sim N(0, 1)$. Back \cite{Back1992} shows that $P^*(t, Q_{[0, t]}) = Q_t$ is an equilibrium, with corresponding $\th^*$ takes the following form:
\bea
\label{bridge1}
\th^*(v; t, Q_{[0,t]}) = {v - Q_t \over 1-t}, \q 0\le t<1,
\eea
 and, for each $v\in \dbR$, the SDE \reff{Qth} with $\th'=\th^*(v;\cd)$  has a unique solution $Q^v$:
\bea
\label{bridge2}
Q^v_t = vt + (1-t)\int_0^t {dB_s\over 1-s},\q \th^{*v}_t := \th^*(v; t, Q^v_{[0,t]}) =v -\int_0^t {dB_s\over 1-s},\q 0\le t<1.
\eea
For completeness, we shall provide a proof of these results in Appendix.\footnote{The proof in \cite{Back1992} relies on a general result by using PDE arguments. Our proof below is specifically for this example,  which is easier and seems new.}  
 We remark that $Q^v$ is a Brownian bridge, and it is clear that $\dbE\big[\int_0^T |\th^{*v}_t|^2dt\big]=\infty$.

(ii)  Under the equilibrium $(P^*, \th^*)$, by \reff{bridge2} we have $S_t = Q^V_t = Vt +  (1-t)\int_0^t {dB_s\over 1-s}$, and $dS_t = \th^{*V}_t dt + dB_t$. Since $S_0=Q^V_0=0$, it is clear that $S$ is Markovian under the equivalent measure $\tilde \dbP$ obtained from applying the Girsanov theorem. Such Markov property is crucial for the PDE approach in the existing literature.

(iii) The above equilibrium is unique among all equilibria $(P^*, \th^*)$ such that either $P^*$ is Markovian as in (ii), cf. \cite{Back1992},  or $P^*$ is linear in $Q$, cf. \cite{Cho2003}. However, the uniqueness among general class of equilibria is not clear. 
\end{rem}

\subsection{The weak formulation}
\label{sect-weak}
The existing literature mainly focuses on the existence of equilibrium by following PDE approach, which requires certain Markovian  structure in the spirit of Remark \ref{rem-Kyleexample} (ii). Inspired by \reff{Qth}, we shall use the weak formulation which is widely used in the literature of stochastic controls and stochastic differential games. 

For this purpose, we fix the canonical space for $Q$:
\bea
\label{QB}
Q \equiv B,\qq \dbF := \si(V)\vee \dbF^Q = \si(V)\vee \dbF^B.
\eea
That is, $Q$ is a $\dbP$-Brownian motion. In this case, the pricing rule $P_t = P(t, Q_{[0, t]})$ and each $\th'_t = \th'(t, Q_{[0,t]})$ are just $\dbF^B$-progressively measurable processes, and when there is no confusion, we will just denote them as $P_t$ and $\th'_t$. Note that under these notations we have $S_t = P_t$. For each $\th'$, we introduce the equivalent probability measure $\dbP^{\th'}$ through the Girsanov theorem:
\bea
\label{Girsanov}
B^{\th'}_t := B_t - \int_0^t \th'_s ds,\q M^{\th'}_t := \exp\Big(\int_0^t \th'_s dB_s - {1\over 2} \int_0^t |\th'_s|^2 ds\Big),\q d \dbP^{\th'}:= M^{\th'}_T d\dbP. 
\eea

We next modify the insider's problem \reff{JPth0}:
 \bea
\label{JPth}
J(P; v, \th') := \dbE^{\dbP^{\th'}}\Big[\int_0^T [ v- P_t] \th'_tdt - \int_0^T f_t(\th'_t) dt \Big].
\eea
Here $f_t(\th'_t) = f(t, Q_{[0,t]}, \th'_t)$ stands for certain cost or risk (e.g. legal risk) for implementing the strategy $\th'$, and as in $P_t$, $\th'_t$ we will often omit the variable $Q\equiv B$. So the insider controls the distribution $\dbP^{\th'}$. Similarly, we may rewrite \reff{Pth0} as:
\bea
\label{Pth}
P^\th_t = P^\th(t, Q_{[0, t]}) = \dbE^{\dbP^\th}\big[V |\cF^Q_t\big],\q 0\le t\le T,\q\mbox{where}\q d \dbP^\th := d\dbP^{\th^v} \mu(dv)
\eea
 is the joint law of $(V, Q)$. By Bayes’ rule, or the Kallianpur-Striebel formula \cite{KS, KS2}, we have, 
\bea
\label{filter}
    P^\th_t=\frac{\int_{\dbR} vM_t^{\th^v}\mu(dv)}{\int_{\dbR}M_t^{\th^v}\mu(dv)}, \quad\dbP-a.s. 
\eea

We now specify some technical conditions. Let $A$ be a domain in certain Euclidean space, $\cA'$ the set of $\dbF^B$-progressively measurable bounded $A$-valued process $\th'$, and $\cA$  the set of admissible controls $\th$ of the insider such that $\th$ is progressively measurable in all variables and $\th^v\in \cA'$ for each $v\in \mbox{Supp}(V)$. Here, to ensure the application of the Girsanov theorem, we assume each $\th'\in \cA'$ is bounded:
\bea
\label{thbound}
|\th'|\le C_{\th'}.
\eea
This assumption is for technical convenience and can be weakened. In particular, we emphasize that $C_{\th'}$ may depend on $\th'$ and we do not require a uniform bound for all $\th'\in \cA'$ (unless $A$ is bounded). Moreover, we assume that
\bea
\label{fcondition}
\dbE\Big[\int_0^T |f_t(0)|^2dt\Big]<\infty,\qq f_t(\th) \ge f_t(0) - C|\th|.
\eea
This ensures that $J(P; v, \th')$ in \reff{JPth} makes sense, with a possible value $-\infty$, but with $J(P; v, 0) \in \dbR$.

\begin{defn}
\label{defn-NE}
We say a pair $(P^*, \th^*)$ is an equilibrium of the model \reff{JPth}-\reff{Pth}, denoted as $(P^*, \th^*) \in \cE$, or simply say an equilibrium of the game, if $P^* = P^{\th^*}$ as in \reff{Pth}, and $\th^{*v} = \mbox{argmax}_{\th'} ~ J(P^*; v, \th')$ as in \reff{JPth}, for all $v\in \mbox{Supp}(V)$.
\end{defn}

 \begin{rem}
 \label{rem-cA}
 In this remark we compare Definitions \ref{defn-NE0} and \ref{defn-NE} when $f\equiv  0$ in \reff{JPth}.
 
(i) For each $\th'\in \cA'$, by Girsanov theorem we know \reff{Qth} has a unique (in law) weak solution, and $J_0(P; v, \th') = J(P; v, \th')$. Then, for $\th^*\in \cA$ and by considering  weak solution of \reff{Qth} in Definition \ref{defn-NE0}, we see that the two definitions are equivalent: $(P^*, \th^*)$ is an equilibrium in the sense of Definition \ref{defn-NE} if and only if  it is an equilibrium in the sense of Definition \ref{defn-NE0}.

(ii) The $\th^*$ in Remark \ref{rem-Kyleexample} does not satisfy the requirement \reff{thbound}, and thus the $(P^*, \th^*)$ there is not an  equilibrium under our definition. However, as pointed out in (i), the only difference between the two definitions is the boundedness (or more generally integrability) of $\th^*$. Moreover, as we will show in Subsection \ref{sect-bridge} later, 
our framework still covers $(P^*,\th^*)$ in certain asymptotic sense.  

(iii)  The Brownian bridge structure in Remark \ref{rem-Kyleexample} does not satisfy the square integrability, and it is hard to establish a general theory even for control problems under such a weak integrability condition. Moreover, if we aim to consider the insider's strategy as general as possible, we may relax  the absolute continuity of $\Th$ further and allow it to be a semimartingale, then the insider's optimal utility becomes $\infty$ and the linear $P^*$ in Remark \ref{rem-Kyleexample} is no longer an equilibrium. See Appendix for the detailed analysis. 

(iv) We should also point out that, when legal risk is considered, namely with the presence of $f$, a large $|\th|$ will expose the insider to a large risk, and thus practically it is also not optimal for the insider to trade too aggressively. See \cite{MaXiaZhang} for more analysis on this.
\end{rem}

\section{FBSDE characterization of equilibria}
\label{sect-FBSDE}
In this section, we characterize all equilibria $\cE$ through a coupled system of forward backward SDEs. For technical convenience, in this section we assume $V$ is discrete: for some $N\ge 2$,
\bea
\label{discrete}
\mbox{Supp}(V) = \{v_1, \cds, v_N\},\q\mbox{with}\q p_i := \dbP(V=v_i) >0,~ i=1,\cds, N.
\eea
In this case, we may naturally denote $\th^i := \th^{v_i}$ for each $\th \in \cA$, and \reff{filter} becomes
\bea
\label{filter-discrete}
 P^\th_t=\frac{\sum_{i=1}^N p_i v_i M_t^{\th^i}}{\sum_{i=1}^N p_i M_t^{\th^i}}, \quad\dbP-a.s. 
\eea
We remark that all the results in this section can be extended to the general case, with the FBSDE system becomes infinite dimensional. 

We first analyze the market marker's problem. Fix $\th\in \cA$ and denote
\bea
\label{Xth}
X^{\th,i}_t := \dbP^\th(V=v_i |\cF_t) = {p_i M^{\th^i}_t\over \sum_{j=1}^N p_j M^{\th^j}_t},\q i=1,\cds, N.
\eea
Then clearly
\bea
\label{Xthproperty}
X^\th_t \in \D_N:= \big\{x \in \dbR^N: 0< x_i <1,~ \sum_{i=1}^N x_i = 1\big\},\qq P^\th_t = \sum_{i=1}^N v_i X^{\th,i}_t.
\eea
Moreover, by \reff{Girsanov}, it follows from the It\^{o} formula that
\bea
\label{XthSDE}
\dis X_t^{\th, i}=p_i+ \int_0^t X^{\th,i}_s\big[\th^i_s - \bar X^\th_s] \big[dB_s - \bar X^\th_s ds\big],\q \mbox{where}\q \bar X^\th_s:= \sum_{j=1}^N \th^j_s X^{\th, j}_s.
\eea
We remark that this SDE system has triple growth in the unknown $X^\th$. However, since $\th\in \cA$ is bounded, it is clear that the $X^\th$ in \reff{Xth} is the unique bounded solution.

We next solve the insider's problem for a given pricing rule $P$. Recall \reff{JPth}, we have $J(P; v_i, \th') = Y^{P, i, \th'}_0$, where $(Y^{P, i, \th'}, Z^{P, i, \th'})$ solves the following linear BSDE:
\bea
\label{BSDEth}
Y^{P,i,\th'}_t = \int_t^T \Big[(v_i- P_s) \th'_s -  f_s(\th'_s) + \th'_s Z^{P,i,\th'}_s\Big]ds - \int_t^T Z^{P,i,\th'}_s dB_s.
\eea
By the standard comparison principle for BSDEs, we have
\bea
\label{YP}
Y^{P,i}_0 = \sup_{\th'}J(P; v_i, \th'),
\eea
where $(Y^{P,i}, Z^{P,i})$ solves the following nonlinear BSDE:
\bea
\label{BSDEP}
Y^{P,i}_t = \int_t^T H_s\big(v_i- P_s + Z^{P,i}_s\big) ds - \int_t^T Z^{P,i}_s dB_s,~\mbox{where}\q H_s(z) := \sup_{\th\in A} [z \th - f_s(\th)].
\eea

\begin{assum}
\label{assum-H}
The function $f$ satisfies \reff{fcondition}; $H$ is differentiable in $z$ such that $\pa_z H$ is bounded and uniformly Lipschitz continuous in $z$, and $H(0)= H_\cd(Q, 0)$ is bounded.
\end{assum}

We now study the equilibria. First, by the envelope theorem,  we see that the optimal argument for the Hamiltonian $H$ is $\pa_z H$. Then the optimal argument of \reff{YP} is 
\bea
\label{envelope}
\th^{P,i}_t := \pa_z H_t(v_i - P_t + Z^{P,i}_t).
\eea
 Plug this into \reff{XthSDE}, and combine with \reff{BSDEP} and \reff{Xthproperty}, we obtain the following coupled system of FBSDEs: for $i=1,\cds, N$, 
\bea
\label{FBSDE}
\left\{\ba{lll}
\dis X_t^{i}=p_i 
+ \int_0^t X^{i}_s\big[\pa_z H_s(\hat Z^i_s) - \bar X_s] \big[dB_s - \bar X_s ds\big],\ms\\
\dis Y^{i}_t = \int_t^T H_s\big(\hat Z^{i}_s\big) ds - \int_t^T Z^{i}_s dB_s,\ms\\
\dis \mbox{where}\q \bar X_s:= \sum_{j=1}^N  X^{j}_s\pa_z H_s(\hat Z^j_s),\q \hat Z^i_s := v_i + Z^i_s- \sum_{j=1}^N v_j X^j_s.
\ea\right.
\eea

\begin{thm}
\label{thm-equilibrium}
Let \reff{discrete} and Assumption \ref{assum-H} hold true. Then $(P^*,\th^*)$ is an equilibrium in the sense of Definition \ref{defn-NE} if and only if 
\bea
\label{NEcharacterization}
P^* = \sum_{i=1}^N v_i X^i,\q \th^*(v_i;\cd) = \pa_z H(\hat Z^i),
\eea
for some solution $(X, Y, Z)$ of the FBSDE \reff{FBSDE} such that $X$ is bounded.
\end{thm}

\begin{proof} If $(P^*,\th^*)\in \cE$ as in Definition \ref{defn-NE}, then from our derivation of \reff{FBSDE} we see that $(X^{\th^*},  Y^{P^*}, Z^{P^*})$ satisfy all the requirements. Thus it suffices to prove the if direction. Assume $(X, Y, Z)$ satisfies \reff{FBSDE} with bounded $X$. We shall show that $(P^*, \th^*)$ defined by \reff{NEcharacterization} is an equilibrium.

First, given $P^* = \sum_{j=1}^N v_j X^j$, by the BSDE in \reff{FBSDE} we see that $(Y, Z)$ satisfies BSDE \reff{BSDEP} with $P=P^*$. Since $\pa_z H$ is bounded and thus $H$ is uniformly Lipschitz continuous in $z$, again for fixed $P^*$, the BSDE \reff{BSDEP} has a unique solution. Then $(Y^{P^*}, Z^{P^*}) = (Y, Z)$, and thus, by \reff{envelope}, the insider's optimal strategy corresponding to $P^*$ is:
\bea
\label{thP*}
\th^{P^*, i} = \pa_z H(v_i - P^* + Z^{P^*,i}) =  \pa_z H\big(v_i - \sum_{j=1}^N v_j X^j + Z^i\big) = \th^*(v_i; \cd).
\eea

On the other hand, given $\th^*(v_i;\cd) = \pa_z H(\hat Z^i)$, by the forward SDE as well as the definition of $\bar X$ in \reff{FBSDE}, we see that $X$ satisfies \reff{XthSDE} with $\th = \th^*$. Since $\pa_z H$ is uniformly Lipschitz continuous in $z$, and by \reff{thbound}, we see that SDE \reff{XthSDE} has a unique bounded solution. Then $X^{\th^*} = X$, and thus, by \reff{Xthproperty} and the third line of \reff{FBSDE},
\beaa
P^{\th^*} = \sum_{i=1}^N v_i X^{\th^*,i} = \sum_{i=1}^N v_i X^i = P^*.
\eeaa
This, together with \reff{thP*}, implies that $(P^*, \th^*)\in \cE$.
 \end{proof}
 
 \begin{rem}
\label{rem-MFG}
Note that $X_t = (X^1_t, \cds, X^N_t)\in \D_N$ stands for the conditional distribution of $V$, under the equilibrium measure $\dbP^{\th^*}$ as in \reff{Xth}.

(i) The FBSDE \reff{FBSDE} shares similar structure as the mean field game system, whose solutions have one to one correspondence with the mean field equilibria, see e.g. \cite{CD1}. In particular, in both systems the backward component corresponds to the optimal value of a player and  to an HJB equation in the Markov case. However, for mean field game system, the forward component is the trajectory of a moving particle and in the Markov case its distribution satisfies a Fokker-Planck equation; while in \reff{FBSDE} $X$ itself stands for a (random) distribution and its law will become a law of random measure. We shall explore these interesting connections in details in future research. 

(ii) The dynamic of $X$ can be viewed as the transport of the conditional distribution of $V$. So our approach also has intrinsic connections with the optimal transport approach for insider problems proposed by \cite{back2021, Bose2024}. We shall leave this also for future research.
\end{rem}

\section{A wellposedness result over small time duration}
\label{sect-smallT}
The wellposedness of the FBSDE \reff{FBSDE} is in general challenging, especially since the diffusion term of the forward SDE involves $Z$. However, the result is easy when $T$ is small.

\begin{thm}\label{thm-smallT}
Let   \reff{discrete} and Assumption \ref{assum-H} hold true. Then there exists $\d>0$, which depends only on the distribution of $V$ in \reff{discrete} and the Lipschitz constant of $H$ and $\pa_zH$, such that the FBSDE \reff{FBSDE} has a unique solution with bounded $X$ whenever $T\le \d$. Consequently, when $T\le \d$, the game has a unique equilibrium $(P^*, \th^*)\in \cE$. 
\end{thm}
We remark that, to the best of our knowledge, this is the first uniqueness result for continuous time Kyle-Back model, without restricting to special types of pricing rules.

\ms
\begin{proof} We first prove the existence. The main difficulty lies in the fact that the forward SDE is not Lipschitz continuous in $(X, Z)$. We shall get around of this difficulty by truncation. To be precise, in light of \reff{Xthproperty} we introduce the truncation function $I(x) := 0 \vee x \wedge 1$. We then consider the following truncated FBSDE:
\bea
\label{FBSDE-truncate}
\left\{\ba{lll}
\dis \cX_t^{i}=p_i+ 
 \int_0^t I(\cX^{i}_s)\big[\pa_z H_s(\hat \cZ^i_s) - \bar \cX_s] \big[dB_s - \bar \cX_sds\big],\ms\\
\dis \cY^{i}_t = \int_t^T H_s\big(\hat \cZ^{i}_s\big) ds - \int_t^T \cZ^{i}_s dB_s,\ms\\
\dis \mbox{where}\q \bar \cX_s:= \sum_{j=1}^N  I(\cX^{j}_s)\pa_z H_s(\hat \cZ^j_s),\q \hat \cZ^i_s := v_i + \cZ^i_s- \sum_{j=1}^N v_j \cX^j_s.
\ea\right.
\eea
Since $H$ and $\pa_z H$ are uniformly Lipschitz continuous, the above FBSDE is uniformly Lipschitz continuous, with the uniform Lipschitz constant depending only on the distribution of $V$ in \reff{discrete} and Lipschitz constant of $H$ and $\pa_zH$. Then, by standard FBSDE literature, see e.g. \cite{Zhang2017}, there exists $\d>0$, depending only on the above model parameters, such that the FBSDE \reff{FBSDE-truncate} has a unique solution $(\cX, \cY, \cZ)$ whenever $T\le \d$. Then it suffices to prove $0<\cX^i<1$, which implies $I(\cX^i) =\cX^i$, and thus $(\cX, \cY, \cZ)$ satisfies \reff{FBSDE}.

To this end, we first note that $I(0)=0$. Then, since $p_i>0$, by standard SDE estimates we derive from the forward SDE in \reff{FBSDE-truncate} that $\cX^i_t >0$ for all $0\le t\le T$ and $i=1,\cds, N$. Next, note that $\sum_{i=1}^N p_i = 1$ and
\beaa
\sum_{i=1}^N I(\cX^{i}_s)\big[\pa_z H_s(\hat \cZ^i_s) - \bar \cX_s] = \sum_{i=1}^N I(\cX^{i}_s)\pa_z H_s(\hat \cZ^i_s) - \sum_{i=1}^N I(\cX^{i}_s) \bar \cX_s = \big(1-\sum_{i=1}^N I(\cX^{i}_s)\big) \bar \cX_s.
\eeaa
Then, the forward SDEs in \reff{FBSDE-truncate} we have
\bea
\label{sumX}
\sum_{i=1}^N \cX^i_t = 1 + 
\int_0^t \big(1-\sum_{i=1}^N I(\cX^{i}_s)\big) \bar \cX_s \big[dB_s - \bar \cX_sds\big].
\eea
Introduce a stopping time 
\beaa
\t := \inf\big\{t\ge 0: \max_{1\le i\le N} \cX^i_t \ge 1\big\}\wedge T.
\eeaa
Since $\max_{1\le i\le N} \cX^i_0 = \max_{1\le i\le N} p_i <1$, we have $\t>0$. For $t\in [0, \t]$, we have $I(\cX^i_t) = \cX^i_t$, and thus \reff{sumX} becomes
\beaa
\sum_{i=1}^N \cX^i_t = 1 + 
 \int_0^t \big(1-\sum_{i=1}^N \cX^{i}_s\big) \bar \cX_s \big[dB_s - \bar \cX_sds\big].
\eeaa
This implies that $\sum_{i=1}^N \cX^{i}_t \equiv 1$, $0\le t \le \t$. Then, since $\cX^i_t >0$, we must have $\cX^i_t <1$ for $t\le \t$, and thus $\t = T$. Consequently, $0<\cX^i_t<1$ for $t\in [0, T]$ and hence $(\cX, \cY, \cZ)$ is a solution to FBSDE \reff{FBSDE}.  

We next prove the uniqueness. Let $(\tilde X, \tilde Y, \tilde Z)$ be an arbitrary solution to \reff{FBSDE} such that $\tilde X$ is bounded. Denote $\th^i_t := \pa_z H(v_i+ \tilde Z^i_t-\sum_{j=1}^N v_j \tilde X^j_t)$. Then $\tilde X$ satisfies SDE \reff{XthSDE}. In particular, by the uniqueness of bounded solution to \reff{XthSDE}, we have $0< \tilde X^i_t <1$, and thus $(\tilde X, \tilde Y, \tilde Z)$ also satisfies the truncated FBSDE \reff{FBSDE-truncate}. Now it follows from the uniqueness of \reff{FBSDE-truncate} that $(\tilde X, \tilde Y, \tilde Z) = (X, Y, Z)$, hence the uniqueness of \reff{FBSDE}. 

Finally, by Theorem \ref{thm-equilibrium}  the game has a unique equilibrium $(P^*, \th^*)$.
\end{proof}

From now on we study the Markov case. Assume
\bea
\label{Markov}
f_t(Q_{[0, t]}; \th) = f(t, \th)\q\mbox{and hence}\q H_t(Q_{[0, t]}, z) = H(t, z),
\eea
where, by abusing the notations, the functions $f, H$ in the right side above are deterministic. Then FBSDE \reff{FBSDE} becomes Markovian, and the following result is immediate.

\begin{prop}
\label{prop-MarkovianX}
Let all the conditions in Theorem \ref{thm-smallT} and \reff{Markov} hold true, and $T\le \d$. Then the solution $X$ is Markovian under $\dbP$, and there exist deterministic functions $u_i(t, x_1, \cds, x_N)$, $i=1,\cds, N$, such that $Y^i_t = u_i(t, X^1_t, \cds, X^N_t)$. 
\end{prop}

\begin{rem}
\label{rem-Master}
Recall Remark \ref{rem-MFG} that $X_t$ stands for the conditional distribution of $V$ under $\dbP^{\th^*}$. We may rewrite the connection between $X$ and $Y$ as $Y^i_t = U(t, v_i, \cL_{V|\cF^B_t})$. Note that $\{v_i\}$ is exactly the support of the distribution $\cL_{V|\cF^B_t}$, so $U$ is a function taking the form $U(t, x, \mu)$, where $\mu\in \cP(\dbR)$, the probability measures on $\dbR$, and $x\in Supp(\mu)$. It will be very interesting to explore its connection with mean field game master equations (cf. \cite{CCD, CDLL, GMMZ}). We shall leave this for future research.
\end{rem}

\begin{prop}
\label{prop-MarkovianS}
Let all the conditions in Theorem \ref{thm-smallT} and \reff{Markov} hold true, and $T\le \d$. Assume further that $N=2$. Then $S\equiv P^*$ is Markovian under $\dbP$. 
\end{prop}
\proof Note that $X^2=1-X^1$ when $N=2$, then the Markov property of $(X^1, X^2)$ clearly implies the Markov property of $X^1$.
By \reff{NEcharacterization} we have $P^*_t = v_1 X^1_t + v_2(1-X^1_t) = (v_1-v_2) X^1_t + v_2$. Since $v_1 - v_2\neq 0$ (otherwise $N=1$ and there is no inside information), then clearly $P^*$ is also Markovian under $\dbP$.
\qed

\begin{rem}
\label{rem-Markov}
(i) Note that we are using weak formulation, so the ``true'' distribution of $V$ for the market maker should be under $\dbP^{\th^*}$. The Markov property of $S=P^*$, however, is under the reference measure $\dbP$, under which $Q$ is a martingale, rather than $\dbP^{\th^*}$.

(ii) When $N>2$, although the random vector $X=(X^1, \cds, X^N)$ is Markovian under $\dbP$, 
typically $P^*_t = \sum_{i=1}^N v_i X^i_t$ is not, see some discussions on this in Appendix below. Since our equilibrium here is unique, then in this case there will be no equilibrium such that the market price $S=P^*$ is Markovian under $\dbP$.

(iii) In the literature, in order to use the PDE approach to construct an equilibrium, one typically assumes the structure that $S_t = P(t, \G_t)$, where $P(t,\cd)$ is a one to one function, $\G$ is the so called weighted total order process satisfying $d\G_t = w(t, \G_t) dQ_t$ for some weight function $w$. Clearly, under the reference measure $P$ such that $Q$ is a Brownian motion (or more generally a Markovian martingale), $\G$ is Markovian, and since $P$ is one to one, we see that $S$ is Markovian under $\dbP$. Therefore, from (ii) we see that the PDE approach cannot succeed in the general setting of Proposition \ref{prop-MarkovianX} when $N\ge 3$. 
\end{rem}

We conclude this section with an example satisfying the conditions in Proposition \ref{prop-MarkovianX}.
\begin{eg}
\label{eg-H}
Consider 
\bea
\label{feg}
\left.\ba{c}
A = [-1, 1],~ f(t, \th) = f(\th) = - \sqrt{1-|\th|^2};\ms\\
\mbox{or}\q A = \dbR, ~ f(\th) = \left\{ \ba{lll} 1- \sqrt{1-|\th|^2},& |\th|\le 1,\\ \infty,& |\th|>1.\ea\right.
\ea\right.
\eea
In both cases, we have
\beaa
H(t,z) = H(z) = \sqrt{1+|z|^2},\q \mbox{with optimal argument}~ \th^* = {z\over \sqrt{1+|z|^2}}.
\eeaa
Note that
\beaa
\pa_z H(z) = {z\over \sqrt{1+|z|^2}},\q \pa_{zz} H(z) = {1\over \sqrt{(1+|z|^2)^3}}.
\eeaa
One can easily verify  Assumption \ref{assum-H} in this case. Consequently, by Theorem \ref{thm-smallT}, there exists $\d>0$ such that the model \reff{JPth}-\reff{Pth}-\reff{feg} has a unique equilibrium when $T\le \d$. 
\end{eg}

\section{Set value of the game}
\label{sect-setvalue}
Inspired by \cite{FRZ2021} and \cite{IseriZhang}, in this section we study the set value of the game, which weakens the notion of equilibrium and focuses on the insider's values, and is particularly effective when there are multiple equilibria. For simplicity we still consider discrete $V$ as in \reff{discrete}.

We first define the raw set value of the problem:
\bea
\label{rawV}
\dbV_0 := \big\{ \big(J(\dbP^*; v_i, \th^*)\big)_{1\le i\le N}: (P^*, \th^*)\in \cE\big\} = \big\{Y_0: ~\mbox{all solution to}~\reff{FBSDE}\big\}\subset \dbR^N.
\eea
That is, $\dbV_0$ is the set of values (not the set of equilibria!) of the insider over all equilibria, which is by nature unique.   Moreover, by viewing the empty set $\emptyset$ as a possible value, we may analyze $\dbV_0$ even when there is no equilibrium. However, $\dbV_0$ relies on true equilibria, and  is hard to analyze mathematically. Indeed, in standard stochastic control problems, the value function is defined as the supremum/infimum value, which is the limit of approximate values over approximate optimal controls, rather than the value over true optimal controls which may not exist. We shall modify  $\dbV_0$ and define the set value $\dbV$ in the same spirit. For this purpose, we first introduce approximate equilibria. Recall \reff{JPth} and \reff{Pth}.

 \begin{defn}
\label{defn-NEe}
For any $\e>0$, we say a pair $(P^\e, \th^\e)$ is an $\e$-equilibrium of the model \reff{JPth}-\reff{Pth}, denoted as $(P^\e, \th^\e) \in \cE_\e$, if 
\bea
\label{the}
&\dis \int_\dbR \Big[\sup_{\th'\in \cA'} J(P^\e; v, \th') - J(P^\e; v, \th^{\e,v})\Big] \mu(d v) \le \e;\\
\label{Pe}
&\dis \dbE^{\dbP}\Big[ \int_0^T\big|P^\e_t - P^{\th^\e}_t\big|^2 dt\Big] \le \e^2.
\eea
\end{defn}
Here we consider the square error for $P^\e$ and the expectation is under $\dbP$, instead of $\dbP^{\th^{\e}}$, for technical reasons. 
We are now ready to define the set value of the game under \reff{discrete}:
\bea
\label{dbV}
&\dis \dbV := \bigcap_{\e>0} \dbV_\e,\q\mbox{where}\\
&\dis \dbV_\e := \Big\{y\in \dbR^N: \sum_{i=1}^N p_i \big|y_i- J(P^\e; v_i, \th^\e(v_i;\cd)\big| \le \e ~ \mbox{for some}~ (P^\e, \th^\e)\in \cE_\e\Big\}.\nonumber
\eea

It is clear that $\dbV_0 \subset \dbV$, and the existence of equilibrium is equivalent to $\dbV_0 \neq \emptyset$. We may say the game has a value when $\dbV \neq \emptyset$. The following result shows that the existence of game value is weaker, and thus more likely, than the existence of equilibrium. 

\begin{prop}
\label{prop-Vcompact}
Let   \reff{discrete} and Assumption \ref{assum-H} hold true. Then $\dbV$ is compact. Moreover, $\dbV\neq \emptyset$ if and only if there exists an $\e$-equilibrium for any $\e>0$.
\end{prop}
\proof We first show that, for $0<\e\le 1$, $\dbV_\e$ is bounded whenever $\cE_\e\neq \emptyset$. Indeed, fix an arbitrary $(P^\e, \th^\e)\in \cE_\e$. By \reff{discrete} $V$ is bounded, then it follows from \reff{Pth} that $P^{\th^\e}$ is bounded, and thus by \reff{Pe} we have $\dbE^{\dbP}\Big[\int_0^T |P^\e_t|^2dt\Big]\le C<\infty$. Applying Assumption \ref{assum-H} on \reff{BSDEP}, this implies further that $|Y^{P^\e,i}_0|\le C$.  Note that, for $i=1,\cds, N$,
\beaa
J(P^\e; v_i, \th^\e(v;\cd)) \le \sup_{\th'\in \cA'} J(P^\e; v_i, \th') = Y^{P^\e,i}_0\le C.
\eeaa
On the other hand, by \reff{the},
\beaa
J(P^\e; v_i, \th^\e(v;\cd)) \ge Y^{P^\e,i}_0 - {\e\over p_i} \ge -C.
\eeaa
So $|J(P^\e; v_i, \th^\e(v;\cd))|\le C$. This clearly implies that $\dbV_\e$ is bounded.

We next show that, denoting by cl the closure, 
\bea
\label{closure}
cl(\dbV_{\e_1}) \subset \dbV_{\e_2},\q\mbox{for any}~ 0<\e_1 <\e_2.
\eea 
Indeed, let $y\in cl(\dbV_{\e_1})$ and $\d := \e_2-\e_1>0$, by definition there exists $y^\d$ and  $\e_1$-equilibrium $(P^{\e_1}, \th^{\e_1})$ such that $|y-y^\d|\le \d$ and $\sum_{i=1}^N p_i \big|y^\d_i- J(P^\e; v_i, \th^\e(v_i;\cd)\big| \le \e_1$. Then
\beaa
\sum_{i=1}^N p_i \big|y_i- J(P^\e; v_i, \th^\e(v_i;\cd)\big| \le \e_1 + \sum_{i=1}^N p_i \big|y_i- y^\d_i\big| \le \e_1 + \d =\e_2.
\eeaa
Clearly $(P^{\e_1}, \th^{\e_1})\in \cE_{\e_1} \subset \cE_{\e_2}$, then $y \in \dbV_{\e_2}$. 

Now it is clear that $\dbV = \bigcap_{\e>0} cl(\dbV_\e)$. By the boundedness of $\dbV_\e$ we see that $ cl(\dbV_\e)$ is bounded and closed, and hence compact. Then one can easily see that $\dbV$ is compact.

Finally, if $\dbV\neq \emptyset$, then clearly $\dbV_\e\neq \emptyset$ and thus there exists $\e$-equilibrium, for any $\e>0$. On the other hand, assume $\dbV_\e\neq \emptyset$ for all $\e>0$. Fix arbitrary $y_\e\in \dbV_\e$. Note that $\dbV_\e$ is decreasing as $\e\downarrow 0$. Then, for $\e_0:=1$, we have $y_\e\in \dbV_{\e_0}\subset cl(\dbV_{\e_0})$ for all $0<\e\le \e_0$. Since $cl(\dbV_{\e_0})$ is compact, there exists $\e_n \downarrow 0$ such that $y_{\e_n} \to y_*$ for some $y_*\in cl(\dbV_{\e_0})$. Note that, for any $\e>0$, $\e_n<\e$ when $n$ is large enough, then $y_{\e_n}\in \dbV_{\e_n}\subset \dbV_\e$.  Thus $y_*\in cl(\dbV_\e)$. Since $\e>0$ is arbitrary, we have $y_*\in \dbV = \bigcap_{\e>0} cl(\dbV_\e)$, that is, $\dbV \neq \emptyset$.
\qed

\subsection{The value at a bridge type equilibrium}
\label{sect-bridge}
We now turn to the example in Remark \ref{rem-Kyleexample}, which is a classical example in the literature. Recall again that, due to the requirement \reff{thbound}, our definition of equilibrium excludes the $(P^*, \th^*)$ there. However, in this subsection we shall show that $J_0(P^*; \cd; \th^*)$ belongs to the set value $\dbV$, and thus, in terms of values instead of equilibria, our set value covers the classical results in the literature where the insider's equilibrium strategy takes the bridge form and is unbounded. We remark that, although we restrict our discussion to the example in Remark \ref{rem-Kyleexample}, we believe our statement holds true in much more general cases. In particular, this implies that $\dbV \neq \emptyset$ in that case.

However, note that the $V$ in Remark \ref{rem-Kyleexample} has continuous distribution, so we shall first extend the notion of set value to the general case where $V$ is not discrete anymore. We now consider the setting in Subsection \ref{sect-standardsetting}. Let  $\dbL^0_V$ denote the set of measurable functions $J:  Supp(V)\to \dbR$. Define $\cE_\e$ still as in Definition \ref{defn-NEe}, and we modify \reff{dbV} as follows:
\bea
\label{dbV2}
&\dis \dbV := \bigcap_{\e>0} \dbV_\e \subset \dbL^0_V,\q\mbox{where}\\
&\dis \dbV_\e := \Big\{J\in \dbL^0_V: \int_{Supp(V)}\big|J(v)- J(P^\e; v, \th^\e(v;\cd)\big|\mu(dv) \le \e ~ \mbox{for some}~ (P^\e, \th^\e)\in \cE_\e\Big\}.\nonumber
\eea

\begin{prop}
\label{prop-bridge}
Consider the setting in Remark \ref{rem-Kyleexample}.  Then $J_0(P^*; \cd, \th^*) \in \dbV$.
\end{prop}
\proof Recall Remark \ref{rem-cA} that the strong formulation for $J_0$ and the weak formulation for $J$ are consistent.  We shall use both formulations in this proof, whichever is more convenient. Moreover, we shall use the properties in Remark \ref{rem-cA} (i) as well as  the arguments in its proof in Appendix below.  To meet the requirement \reff{thbound}, we truncate $\th^*$ as in \reff{thR} below. Then $|\th^{R,v}_t| \le [|v|+R] R$ and thus $\th^R\in \cA$. We shall show that  $(P^*, \th^R)$ is an $\e_R$-equilibrium, with $\e_R = C R^{-{1\over 2}}$ for some appropriate constant $C$.
 
  We first verify \reff{Pe}, by using the weak formulation.  By \reff{PthRP*} below we have $P^{\th^R}_t = B_{\t_R\wedge t}$. 
Since $P^*_t = Q_t$ and $Q\equiv B$ is a Brownian motion under $\dbP$, then
\beaa
&&\dbE\Big[\int_0^1 |P^*_t - P^{\th^R}_t |^2dt \Big]= \dbE\Big[\int_0^1 |B_t- B_{\t_R \wedge t}|^2dt\Big] = \dbE\Big[\int_{\t_R}^1 |B_t- B_{\t_R}|^2dt\Big] \\
&&= \dbE\Big[\int_{\t_R}^1\dbE\big[ |B_t- B_{\t_R}|^2\big|\cF_{\t_R}\big]dt\Big]= \dbE\Big[\int_{\t_R}^1 (t- \t_R)dt\Big] ={1\over 2} \dbE\big[(1-\t_R)^2\big]
\eeaa
Note that, for $t> R^{-1}$,
\beaa
\dbP(1-\t_R \ge t)  = \dbP(\t_R \le 1-t) \le   \dbP\big(\sup_{0\le s\le 1-t}|B_s|  \ge R\big) \le {1\over R} \dbE\big[\sup_{0\le s\le 1-t}|B_s|\big]\le {C\over R},
 \eeaa
 and for $t\le R^{-1}$, $\dbP(1-\t_R \ge t)  = \dbP(1-\t_R \ge R^{-1})\le 1$.  Then
 \beaa
 \dbE\Big[\int_0^1 |P^*_t - P^{\th^R}_t |^2dt \Big]= \int_0^1 t\dbP(1-\t_R \ge t) dt \le \int_0^{R^{-1}} tdt + \int_{R^{-1}}^1 {C\over R} dt  \le CR^{-1}.
\eeaa
This verifies \reff{Pe} for $\e_R = CR^{-{1\over 2}}$.

We next verifies \reff{the}. It is more convenient to use the strong formulation in this case.  For any $v$, denote $Q^v_t := Q^{\th^{*v}}_t$, $\hat \th^v_t:= \th^*(v; t, Q^v_{[0, t]})$, and $\t_R^v := \t_R(Q^{v})$. By \reff{bridge3} and \reff{bridge4} in the Appendix we have
\bea
\label{maxJ-J}
 \eta&:=& \int_\dbR \Big[\sup_{\th'\in \cA'} J(P^*; v, \th') - J(P^*; v, \th^R(v;\cd))\Big] \mu(d v)\\
&=& \dbE\int_\dbR \Big[{v^2+1\over 2} - \big[v Q^{\th^R}_1 - {|Q^{\th^R}_1|^2 -1\over 2}\big]\Big] \mu(d v) = {1\over 2}\dbE \int_\dbR  |Q^{\th^R}_1-v|^2 \mu(d v).\nonumber
 \eea
 By \reff{bridge2}  we have
 \bea
 \label{QR-v}
&& \dbE[|Q^{\th^R}_1-v|^2] = \dbE\Big[ \big|v \t_R^v -  (1-\t_R^v) \int_0^{\t_R^v}\frac{dB_r}{1-r} - v\big|^2\Big]\nonumber\\
 &&\le C\dbE\Big[|v|^2 (1-\t_R^v)^2 + (1-\t^v_R)^2 \big(\int_0^{\t^v_R}\frac{dB_r}{1-r}\big)^2\Big]\nonumber\\
 &&\le C|v|^2 \dbE\big[(1-\t^v_R)^2\big] + C\Big(\dbE\big[(1-\t^v_R)^4\big]\Big)^{1\over 2} \Big(\dbE\big[\big(\int_0^{\t^v_R}\frac{dr}{(1-r)^2}\big)^2\big]\Big)^{1\over 2}\nonumber\\
 &&\le C|v|^2 \dbE\big[(1-\t^v_R)^2\big] + C\Big(\dbE\big[(1-\t^v_R)^4\big]\Big)^{1\over 2} \Big(\dbE\big[{1\over (1-\t_R^v)^2}\big]\Big)^{1\over 2}\nonumber\\
 &&\le C|v|^2 \dbE\big[(1-\t^v_R)^2\big] + CR\Big(\dbE\big[(1-\t^v_R)^4\big]\Big)^{1\over 2},
 \eea
where the last inequality used the fact that $1-\t_R \ge R^{-1}$. 
Now for $|v|\le {R\over 2}$ and $t > R^{-1}$, by the first equality in \reff{bridge5} below we have
 \beaa
&& \dbP(1-\t^v_R \ge t)  = \dbP(\t^v_R \le 1-t) \le   \dbP\big(\sup_{0\le s\le 1-t}|B_s| +\int_0^{1-t} \big|\int_0^s {dB_r\over 1-r}\big|ds \ge {R\over 2}\big) \\
&&\le {C\over R^4}\dbE\Big[\sup_{0\le s\le 1}|B_s|^4 + \Big(\int_0^{1-t} \big|\int_0^s {dB_r\over 1-r}\big|ds\Big)^4\Big]\\
&&\le {C\over R^4} \Big[1 + \int_0^{1-t} \big(\int_0^s {dr\over (1-r)^2}\big)^2ds\Big]\le {C\over R^4} \Big[1 + \int_0^{1-t} {ds\over (1-s)^2}\Big]\\
&& \le {C\over R^4}[1+ t^{-1}] \le {C\over R^3}.
 \eeaa
Thus, for $|v|\le {R\over 2}$,
\beaa
\dbE[(1-\t^v_R)^4] = 4\dbE\Big[\int_0^1 t^3 \dbP(1-\t_R \ge t) dt \Big] \le 4 R^{-4} + 4\dbE\Big[\int_{R^{-1}}^1  \dbP(1-\t_R \ge t) dt \Big] \le C R^{-3}.
\eeaa
Then, by \reff{QR-v},
\bea
\label{QR-v1}
\dbE[|Q^{\th^R}_1-v|^2]  \le C|v|^2 R^{-{3\over 2}} + CR(R^{-3})^{1\over 2} =  C|v|^2 R^{-{3\over 2}}  + CR^{-{1\over 2}},\q |v|\le {R\over 2}. 
\eea
Moreover, for $|v|>{R\over 2}$, by \reff{QR-v} we have 
\bea
\label{QR-v2}
\dbE[|Q^{\th^R}_1-v|^2]  \le C|v|^2+ CR. 
\eea
Since $V\sim N(0,1)$, then clearly $\int_{|v|> {R\over 2}} [|v|^2 + R] \mu(dv) \le CR^{-{1\over 2}}$. 
Plug \reff{QR-v1} and \reff{QR-v2} into \reff{QR-v}, then by \reff{maxJ-J}  we have: 
\beaa
\eta \le C \int_{|v|\le {R\over 2}} [|v|^2 R^{-{3\over 2}}  + R^{-{1\over 2}}] \mu(dv) + C\int_{|v|> {R\over 2}} [|v|^2 + R] \mu(dv)\le CR^{-{1\over 2}} = \e_R.
 \eeaa
 This verifies \reff{the}, and thus $(P^*, \th^R)$ is an $\e_R$-equilibrium.
  
  Finally,  note that $J_0(P^*; v, \th^{*,v}) = \sup_{\th'\in \cA'}J_0(P^*; v, \th')$. Then 
    \beaa
 \int_\dbR \Big|J_0(P^*; v, \th^{*v})-J_0(P^*; v, \th^{R,v})\Big| \mu(dv) = \eta \le \e_R.
  \eeaa
 This implies that $J_0(P^*;\cd, \th^*) \in \dbV_{\e_R}$ for all $R>1$. Since $\e_R\to 0$ as $R\to\infty$, we obtain $J_0(P^*;\cd, \th^*)\in \dbV$.
  \qed

\section{A level set characterization for the set value}
\label{sect-levelset}
Inspired by \cite{FRZ2021}, in this section we provide a level set characterization for the set value $\dbV$, which can also be viewed as a duality result. See also related works \cite{MaYong, ST} in this spirit. Such a characterization could particularly useful for numerical computation of the set value. For simplicity again we consider the discrete case \reff{discrete}.

Recall \reff{XthSDE} and \reff{BSDEP}. Fix $p =(p_1,\cds, p_N)$ and $y\in \dbR^N$. For any $\th\in \cA$ and $Z=(Z^1,\cds, Z^N)\in (\dbL^2(\dbF^B))^N$, introduce the forward processes: denoting $\th^i:= \th^{v_i}$,
\bea
\label{FSDE}
\left.\ba{lll}
\dis X_t^{p, \th,  i}=p_i+  \int_0^t X^{p,  \th,  i}_s\big[\th^i_s - \bar X^{p,  \th}_s] \big[dB_s - \bar X^{p, \th}_s ds\big],\ms\\
\dis Y^{p,y, \th, Z, i}_t = y_i- \int_0^t H_s\big(\hat Z^{p, \th, Z, i}_s\big) ds + \int_0^t Z^{i}_s dB_s,\ms\\
\dis \mbox{where}\q \bar X^{p, \th}_s:= \sum_{j=1}^N  \th^j_s X^{p, \th,  j}_s,~ \hat Z^{p,  \th, Z, i}_s := v_i + Z^i_s- \sum_{j=1}^N v_j X^{p, \th,  j}_s.
\ea\right.
\eea
Note that $X^{p,\th}$ identifies with the $X^\th$ in \reff{XthSDE}. Since $\th$ is bounded, following the arguments in the proof of Theorem \ref{thm-smallT}, one can show that \reff{FSDE} has a unique solution with $X_t^{p,  \th}\in \D_N$. We then consider the following standard optimization problem with control $(\th, Z)$:
\bea
\label{W}
\left.\ba{c}
\dis W(0,p,y) := \inf_{\th\in \cA, Z\in (\dbL^2(\dbF^B))^N} \dbE^\dbP\Big[\sum_{i=1}^N |Y^{p,y,\th, Z,i}_T|^2 + \sum_{i=1}^N\int_0^T  |h^{p,\th, Z,i}_s|^{4\over 3}ds \Big],\ms\\
\dis\mbox{where}~ h^{p,\th, Z,i}_s := H_s(\hat Z^{p, \th, Z, i}_s) - \hat Z^{p,  \th, Z, i}_s \th^i_s + f_s(\th^i_s).
\ea\right.
\eea
Here, the power ${4\over 3}$ is for technical reason, see \reff{Ehest} below, and can be replaced with any order strictly between $1$ and ${3\over 2}$. 
Note that, if $(P^*, \th^*)$ is an equilibrium with corresponding solution $(X, Y, Z)$ to FBSDE \reff{FBSDE}, then $Y_0 \in \dbV$ and $W(0, p, Y_0)=0$. Our duality result characterizes $\dbV$ by the zero level set of $W$. For technical reason, we assume $A$ is bounded, namely all $\th\in \cA$ are uniformly bounded. Note that, under Assumption \ref{assum-H}, by \reff{envelope} we see that all equilibria $\th^*$ are uniformly bounded, so this boundedness assumption on $A$ is not that serious. 

\begin{thm}
\label{thm-duality}
Let  \reff{discrete} and Assumption \ref{assum-H} hold true, and assume the control set $A$ is bounded. Then, for fixed $p$, 
\bea
\label{duality}
\dbV = \dbN(0, p):= \big\{y\in \dbR^N: W(0, p, y) = 0\big\}.
\eea
\end{thm}
\proof We proceed in two steps. Denote
\bea
\label{cJi}
\cJ_i(p,y,\th, Z):= \dbE^\dbP\Big[ |Y^{p,y,\th, Z, i}_T|^2 +\int_0^T |h^{p, \th, Z, i}_s|^{4\over 3}ds \Big].
\eea

\no{\bf Step 1.} We first prove "$\supset$". Fix $y\in\dbN(0,p)$. For any $\e>0$, there exist $(\th^\e, Z^\e)$ such that, 
\bea
\label{supset1}
\cJ_i(p,y,\th^\e, Z^\e) \le \e^2,\q i=1,\cds, N.
\eea
Denote $X^\e:= X^{p,  \th^\e}$, $Y^\e:= Y^{p, y, \th^\e, Z^\e}$, $\bar X^\e := \bar X^{p,  \th^\e}$, $\hat Z^\e:= \hat Z^{p, \th^\e, Z^\e}$, $h^{\e,i} := h^{p, \th^\e, Z^\e, i}$, and recall $P^{\th^\e}_t := \sum_{i=1}^N v_i X^{\e,i}_t$. We shall show that $(P^{\th^\e}, \th^\e)$ is an $C\e$-equilibrium for some constant $C>0$.

Indeed, \reff{Pe} is automatically true. We now fix $P^\e:= P^{\th^\e}$ and verify \reff{the}. Recall \reff{BSDEth}, \reff{YP}, and \reff{BSDEP}, we have
\bea
\label{supset2}
J(P^\e; v_i, \th^{\e,i}) = \tilde Y^{\e,i}_0,\q \sup_{\th'\in \cA'} J(P^\e; v_i, \th') = \cY^{\e,i}_0,
\eea
where, $(\tilde Y^{\e,i}, \tilde Z^{\e,i}) := (Y^{P^\e, i, \th^{\e,i}}, Z^{P^\e, i, \th^{\e,i}})$ and $(\cY^\e, \cZ^\e) := (Y^{P^\e}, Z^{P^\e})$ satisfy BSDEs: 
\bea
\label{tildeBSDEc}
\left.\ba{c}
\dis \tilde Y^{\e,i}_t = \int_t^T \Big[\hat{\tilde Z}^{\e,i}_s \th^{\e,i}_s -  f_s(\th^{\e,i}_s) \Big]ds - \int_t^T \tilde Z^{\e,i}_s dB_s,\q \hat{\tilde Z}^{\e,i}_s:= v_i- P^\e_s + \tilde Z^{\e,i}_s;\\
\dis \cY^{\e,i}_t = \int_t^T H_s(\hat\cZ^{\e,i}_s)ds - \int_t^T \cZ^{\e,i}_s dB_s,\q \hat\cZ^{\e,i}_s := v_i - P^\e_s + \cZ^{\e,i}_s.
\ea\right.
\eea
Denote $\D \tilde Y^\e := Y^\e -  \tilde Y^\e$,  $\D \tilde Z^\e :=  Z^\e - \tilde Z^\e$, $\D \cY^\e := Y^\e - \cY^\e$,  $\D \cZ^\e := Z^\e - \cZ^\e$. Then, by the above BSDEs and \reff{FSDE}, for $i=1,\cds, N$ we have
\bea
\label{DtildeBSDEc}
\left.\ba{c}
\dis \D \tilde Y^{\e,i}_t = Y^{\e,i}_T + \int_t^T \big[H_s(\hat Z^{\e,i}_s) - \hat{\tilde Z}^{\e,i}_s \th^{\e,i}_s +  f_s(\th^{\e,i}_s)\big]ds -\int_t^T \D \tilde Z^{\e,i}_s dB_s;\\
\dis \D \cY^{\e,i}_t = Y^{\e,i}_T + \int_t^T \big[H_s(\hat Z^{\e,i}_s) - H_s(\hat \cZ^{\e,i}_s)\big]ds -\int_t^T \D \cZ^{\e,i}_s dB_s.
\ea\right.
\eea
Note that $\hat Z^\e_s-\hat \cZ^\e_s = \D \cZ^\e_s$. Since $H$ is uniformly Lipschitz continuous, by standard BSDE estimates and \reff{supset1} we obtain from the second BSDE above that
\bea
\label{DcYest}
|\D \cY^{\e,i}_0| \le C\Big(\dbE\big[|Y^{\e,i}_T|^2\big]\Big)^{1\over 2} \le C\Big({\e^2\over p_i}\Big)^{1\over 2} \le C\e.
\eea
Moreover, note that
\beaa
H_s(\hat Z^{\e,i}_s) - \hat{\tilde Z}^{\e,i}_s \th^{\e,i}_s +  f_s(\th^{\e,i}_s) = \th^{\e,i}_s\D \tilde Z^{\e,i}_s+ h^{\e,i}_s. 
\eeaa
Then from the first BSDE in \reff{DtildeBSDEc} we have 
$
\D \tilde Y^{\e,i}_0 = \dbE\Big[ M^{\th^{\e,i}} \big[Y^{\e,i}_T +  \int_0^T   h^{\e,i}_s ds\big]\Big],
$
and thus
\beaa
\big|\D \tilde Y^{\e,i}_0\big| = \Big(\dbE\big[|M^{\th^{\e,i}}|^2\big]\Big)^{1\over 2}\Big(\dbE\big[|Y^{\e,i}_T|^2\big]\Big)^{1\over 2} + \Big(\dbE\big[|M^{\th^{\e,i}}|^4\big]\Big)^{1\over 4}\Big(\dbE\big[\int_0^T | h^{\e,i}_s|^{4\over 3} ds\big]\Big)^{3\over 4}.
\eeaa
Since $A$ is bounded, and thus $\th^\e$ is uniformly bounded. Then $\dbE\big[|M^{\th^{\e,i}}|^4\big]\le C$, and therefore 
\bea
\label{supset4}
\big|\D \tilde Y^{\e,i}_0\big|  \le C[\e + \e^{3\over 2}] \le C\e.
\eea
This, together with \reff{supset2}  and \reff{DcYest}, implies that
\beaa
0\le \sup_{\th'\in \cA'} J(P^\e; v_i, \th')  - J(P^\e; v_i, \th^{\e,i}) = \D \tilde Y^{\e,i}_0 - \D \cY^{\e,i}_0 \le C\e,
\eeaa
and thus 
\bea
\label{supset3}
\sum_{i=1}^N p_i \Big[\sup_{\th'\in \cA'} J(P^\e; v_i, \th')  - J(P^\e; v_i, \th^{\e,i})\Big]  \le C\e.
\eea
This verifies \reff{the} for $\th^\e$, and hence $(P^\e, \th^\e)$ is an $C\e$-equilibrium.

Moreover, by \reff{supset4} we have
\beaa
\big|y_i - J(P^\e; v_i, \th^{\e,i})\big| =\big|Y^{\e,i}_0-\tilde Y^{\e,i}_0\big|= |\D \tilde Y^{\e,i}_0| \le C\e.
\eeaa
This implies that $y\in \dbV_{C\e}$. Since $\e>0$ is arbitrary, we have $y\in \dbV$. That is, $\dbV \supset \dbN(0, p)$.

\ms
\no{\bf Step 2.} We next prove "$\subset$". Fix $y\in \dbV$. For any $\e>0$, there exists $(P^\e, \th^\e)\in \cE_\e$ such that $\sum_{i=1}^N p_i \big|y_i- J(P^\e; v_i, \th^{\e,i})\big| \le \e$, where $\th^{\e,i} := \th^\e(v_i;\cd)$. 

Let $(\tilde Y^\e, \tilde Z^\e)$ and $(\cY^\e, \cZ^\e)$ be the same as in Step 1, but corresponding to $(P^\e, \th^\e)$ here. In particular, \reff{supset2} and \reff{tildeBSDEc} remain true. We shall use $(\th^\e, \cZ^\e)$ as an approximate optimal control for $W(0, p, y)$, namely we shall estimate 
\bea
\label{cJiest}
\cJ_i(p, y, \th^\e, \cZ^\e)  := \dbE^{\dbP}\Big[ |Y^{p,y,\th^\e, \cZ^\e, i}_T|^2 +\int_0^T |h^{p,\th^\e, \cZ^\e, i}_s|^{4\over 3}ds \Big].
\eea
Denote $h^{\e}_s:= h^{p,\th^\e, \cZ^\e}_s$. We emphasize that, by \reff{W} $h^{\e, i}_s$ relies on $\hat Z^{\e, i}_s:= v_i - P^{\th^\e}_s + \cZ^{\e,i}_s$, which is different from the $\hat \cZ^\e$ in \reff{tildeBSDEc}. In particular, we have
\bea
\label{hatZ-Z}
\hat \cZ^{\e,i}_s- \hat Z^{\e,i}_s = P^{\th^\e}_s - P^\e_s.
\eea

Denote $\D Y^\e := \cY^\e - \tilde Y^\e$, $\D Z^\e:= \cZ^\e - \tilde Z^\e$. Then, by \reff{the},
\bea
\label{initialest}
\D Y^{\e,i}_t\ge 0, \q \sum_{i=1}^N p_i \D Y^{\e,i}_0 \le \e,\q \mbox{and thus}\q \sum_{i=1}^N p_i|y_i- \cY^{\e,i}_0| \le 2\e.
\eea
 By \reff{tildeBSDEc} and \reff{W} we have 
\beaa
&&\dis \D Y^{\e,i}_t =\int_t^T \Big[H_s(\hat\cZ^{\e,i}_s) - \hat{\tilde Z}^{\e,i}_s \th^{\e,i}_s + f_s(\th^{\e,i}_s)\Big]ds - \int_t^T \D Z^{\e,i}_s  dB_s \\
&&=\int_t^T \Big[\th^{\e,i}_s \D Z^{\e,i}_s + H_s(\hat \cZ^{\e,i}_s) - H_s(\hat Z^{\e,i}_s) - [\hat \cZ^{\e,i}_s- \hat Z^{\e,i}_s]\th^{\e,i}_s  + h^{\e,i}_s\Big]ds - \int_t^T \D Z^{\e,i}_s dB_s.
\eeaa
Then, by uniform boundedness of $\th^\e$ and $\pa_z H$, and by \reff{hatZ-Z}, we have
\beaa
\dis {\e\over p_i} &\ge& \D Y^{\e,i}_0 = \dbE\Big[M^{\th^{\e,i}}_T\big[ \int_0^T \big[H_s(\hat \cZ^{\e,i}_s) - H_s(\hat Z^{\e,i}_s) -[\hat \cZ^{\e,i}_s- \hat Z^{\e,i}_s]\th^{\e,i}_s  + h^{\e,i}_s\big] ds\big]\Big]\\
 &\ge& \dbE\Big[ M^{\th^{\e,i}}_T\int_0^T h^{\e,i}_t dt \Big] - C\dbE\Big[ M^{\th^{\e,i}}_T\int_0^T |P^{\th^\e}_t - P^\e_t|\big] dt\Big].
 \eeaa
 Again since $\th^\e$ is uniformly bounded, then by \reff{Pe} we have
\beaa
\dbE\Big[ M^{\th^{\e,i}}_T\int_0^T |P^{\th^\e}_t - P^\e_t|\big] dt\Big] \le \Big(\dbE\big[|M^{\th^{\e,i}}|^2\big]\Big)^{1\over 2}\Big(\dbE\big[\int_0^T |P^{\th^\e}_t - P^\e_t|^2\big] dt\big]\Big)^{1\over 2} \le C\e,
\eeaa
Thus
\beaa
\dbE\Big[ M^{\th^{\e,i}}_T\int_0^T h^{\e,i}_t dt \Big] \le C\e.
\eeaa
Since $V$ is bounded, $P^{\th^\e}$ is bounded, and thus by \reff{Pe} we have $\dbE\big[\int_0^T |P^\e_t|^2dt\big] \le C$. Since $\th^\e$ and $\pa_z H$ are uniformly bounded, by \reff{tildeBSDEc} we have
\beaa
\dbE\Big[\int_0^T |\cZ^\e_t|^2dt\Big] \le C,\q\mbox{and thus}\q \dbE\Big[\int_0^T |h^{\e,i}_t|^2dt\Big] \le C,
\eeaa
Then, noting that $\dbE\big[|M^{\th^{\e,i}}|^{-2}\big]\le C$,
\bea
\label{Ehest}
 &&\dbE\Big[\int_0^T |h^{\e,i}_s|^{4\over 3}ds\Big]  = \dbE\Big[(M_T^{\th^{\e,i}})^{-{1\over 3}}\int_0^T |h^{\e,i}_s| \times |M_T^{\th^{\e,i}}h^{\e,i}_s|^{1\over 3}ds\Big]\nonumber\\
 &&\le \Big(\dbE\big[|M^{\th^{\e,i}}|^{-2}\big]\Big)^{1\over 6}\Big(\dbE\big[\int_0^T |h^{\e,i}_s|^2dt\big]\Big)^{1\over 2} \Big(\dbE\big[ M^{\th^{\e,i}}_T\int_0^T h^{\e,i}_t dt \big]\Big)^{1\over 3} \le C \e^{1\over 3}.
\eea

Moreover, note that
\beaa
&&\dis Y^{\e,i}_t:= Y^{p,y, \th^\e, \cZ^\e, i}_t = y_i- \int_0^t H_s\big(v_i - P^{\th^\e}_s + \cZ^{\e,i}_s\big) ds + \int_0^t \cZ^{\e, i}_s dB_s;\\
&&\dis \cY^{\e,i}_t = \cY^{\e,i}_0 - \int_0^t H_s\big(v_i - P^\e_s + \cZ^{\e,i}_s\big) ds + \int_0^t \cZ^{\e, i}_s dB_s
\eeaa
Since $\cY^{\e,i}_T=0$ and $\pa_z H$ is bounded, by standard SDE estimates and \reff{Pe} we have
\beaa
\dbE\Big[|Y^{\e,i}_T|^2\Big] = \dbE\Big[|Y^{\e,i}_T- \cY^{\e,i}_T|^2\Big] \le C \dbE\Big[|y_0-\cY^{\e,i}_0|^2 + \int_0^T |P^{\th^\e}_s - P^\e_s|^2ds\Big] \le C\e^2.
\eeaa
This, together with \reff{Ehest}, implies that $\cJ_i(p, y, \th^\e, \cZ^\e)\le C\e^{1\over 3}$. Then $W(0, p, y) \le  C\e^{1\over 3}$. Since $\e>0$ is arbitrary, we obtain $W(0, p, y) = 0$. Thus $\dbV\subset \dbN(0, p)$.
\qed

\begin{rem}
\label{rem-W}
(i) In the Markov case \reff{Markov}, we may extend $W$ to $W(t,x,y)$ for $t\in [0, T]$. By \reff{FSDE} and \reff{W} it is clear that $W$ is the unique viscosity solution of the following standard HJB equation: denoting $v := (v_1,\cds, v_N)\in \dbR^N$,
\bea
\label{WHJB}
\left\{\ba{lll}
\dis \pa_t W + \sup_{\th\in A^N, z\in \dbR^N} \Big[{1\over 2} \sum_{i,j=1}^N \big[x_i x_j(\th_i- \th \cd x)(\th_j-\th\cd x) \pa_{x_i x_j} W+ x_i(\th_i-\th\cd x) z_j \pa_{x_iy_j}W \\
\dis \qq  + z_i z_j \pa_{y_i y_j}W\big]  + \sum_{i=1}\big[-x_i (\th_i-\th\cd x) (\th \cd x) \pa_{x_i} W  + z_i \pa_{y_i} W \\
\dis \qq + |H(t, v_i + z_i - v\cd x) - \th_i ( v_i + z_i - v\cd x) + f(t,\th_i)|^{4\over 3}\Big] =0;\\
\dis W(T, x, y) = |y|^2.
\ea\right.
\eea
In the general non-Markov case, similar result holds by considering  path dependent PDEs, see \cite{Zhang2017}.

(ii) Recall Theorem \ref{thm-equilibrium} that a (strong) solution of FBSDE \reff{FBSDE} corresponds to an equilibrium and thus $W(0, p, Y_0)=0$. We may also consider weak solutions of FBSDE \reff{FBSDE}, as proposed in \cite{MaZhang, MaZhangZheng, WangZhang}. That is, we allow $Z$ to be adapted to a filtration larger than $\dbF^B$, and  the BSDE in \reff{FBSDE} contains an additional martingale $\cM$ orthogonal to $B$:
\bea
\label{mgorthogonal}
Y^{i}_t = \int_t^T H_s\big(\hat Z^{i}_s\big) ds - \int_t^T Z^{i}_s dB_s + \cM_T - \cM_t.
\eea 
Correspondingly, we rewrite the second equation in \reff{FSDE} as:
\beaa
Y^{x,y, \th, Z, i}_t = y_i- \int_0^t H_s\big(\hat Z^{x, \th, Z, i}_s\big) ds + \int_0^t Z^{i}_s dB_s + \cM_t,
\eeaa
and define $\tilde W(0,x,y)$ as in \reff{W} with controls $(\th, Z, \cM)$. One can show that $\tilde W$ satisfies the same PDE \reff{WHJB} (or corresponding path dependent PDE). Then, by the uniqueness of the PDE, we have $\tilde W = W$.  Consequently, if FBSDE \reff{FBSDE} has a weak solution $(X, Y, Z, \cM)$, then $W(0, p, Y_0) = \tilde W(0, p, Y_0) = 0$, and therefore $Y_0\in \dbV$. That is, if the FBSDE \reff{FBSDE} has a strong solution, then the game has an equilibrium, while if the FBSDE \reff{FBSDE} has a weak solution, then the game has a value: $\dbV \neq \emptyset$. We leave the details to the interested readers.
\end{rem}

\section{Appendix}
\label{sect-Appendix}

\no{\bf Proof of Remark \ref{rem-Kyleexample} (i).} We proceed in three steps.

\no{\bf Step 1.} We first show that, given $P^*$, for each $v\in Supp(V)$, $\th^{*v}$ is the insider's optimal strategy. We shall use the strong formulation. Indeed, for any $\th'$, note that $Q^{\th'}_t = \int_0^t \th'_s ds + B_t$ and recall $T=1$. Then, given $P^*_t(Q_{[0,t]}) = Q_t$,  by applying the It\^{o} formula we have:
\bea
\label{bridge3}
J_0(P^*; v, \th') = \dbE\Big[\int_0^T (v- Q^{\th'}_t) d Q^{\th'}_t\Big] =  \dbE\Big[v Q^{\th'}_T - {|Q^{\th'}_T|^2 - 1\over 2}\Big].
\eea
By the pointwise optimization in the right side above we see that the optimal argument $Q^*_T = v$, which can be achieved by the $\th^{*v}$ in \reff{bridge1}, with corresponding value
\bea
\label{bridge4}
J_0(P^*; v, \th^{*v}) = \sup_{\th'\in \cA'} J_0(P^*; v, \th')  = {v^2+1\over 2}.
\eea
 Thus $\th^*$ is the insider's optimal strategy corresponding to the $P^*$.

\no{\bf Step 2.} We next verify  \reff{bridge2}. Denote $Q^v_t := Q^{\th^{*v}}_t$ and $\hat \th^v_t:= \th^*(v; t, Q^v_{[0, t]})$, $t<1$, we have
 $$
 (1-t)\hat \th^v_t=v-Q_t^v=v-\int_0^t \hat \th^v_sds-B_t.
 $$ 
 Differentiating both sides we get:
    $$
    (1-t)d\hat \th^v_t-\hat \th^v_tdt=-\hat \th^v_tdt-dB_t.
    $$
Note that $\hat \th^v_0=v$, then we have $\hat \th^v_t=v-\int_0^t\frac{dB_s}{1-s}$, and thus
   \bea
   \label{bridge5}
 Q^v_t=vt-\int_0^t\int_0^s\frac{dB_r}{1-r}ds+B_t= vt-\int_0^t\int_r^t {ds\over 1-r} dB_r +B_t=vt- (1-t)\int_0^t\frac{dB_r}{1-r}.
    \eea
 
\no{\bf Step 3.} We finally show that $P^* = P^{\th^*}$. This, together with Step 1, implies that $(P^*, \th^*)$ is an equilibrium and hence completes the proof.

We shall first use the weak formulation, in particular \reff{filter}, to verify \reff{Pth0}. For any $t<1$, by \reff{Girsanov} one may easily get
\beaa
&\dis  M^{\th^{*v}}_t:= \exp\Big(-{t\over 2(1-t)} v^2 + {\xi_t\over 1-t} v - \zeta_t\Big),\q\mbox{where}\\
&\dis  \xi_t :=  \int_0^t {dB_s\over 1-s} + \int_0^t {B_s\over (1-s)^2}ds = {B_t\over 1-t},\qq \zeta_t:= \int_0^t {B_s\over 1-s} dB_s +{1\over 2} \int_0^t {B_s^2\over (1-s)^2} ds.
\eeaa
Since $V\sim N(0,1)$, then
\beaa
 M^{ \th^{*v}}_t \mu(dv) = {1\over \sqrt{2\pi}} e^{-{1\over 2(1-t)} v^2 + {B_t\over 1-t} v - \zeta_t} dv =  {1\over \sqrt{2\pi}} e^{-{1\over 2(1-t)} (v - B_t)^2 + {B_t^2\over 2(1-t)} - \zeta_t} dv.
 \eeaa
 This implies that $\int_\dbR \big(v-B_t\big) M^{\th^{*v}}_t \mu(dv)=0$, and thus
\bea
\label{vM=BM}
 \int_\dbR v M^{\th^{*v}}_t \mu(dv)  =  B_t\int_\dbR  M^{\th^{*v}}_t \mu(dv).
\eea

However, in order to apply the Girsanov theorem,\footnote{The Kallianpur-Striebel formula \reff{filter} actually holds under weaker conditions, see \cite{KS, KS2}. However, under those conditions we don't have the equivalence between strong formulation and weak formulation. To avoid the possible confusion, we require the conditions for Girsanov theorem here.}  we need to truncate $\th^*$ first: 
\bea
\label{thR}
\th^R(v; t, Q_{[0,t]}) = \th^{*}(v; t, Q_{[0,t]}) \1_{\{\t_R \ge t\}}, ~\mbox{where}~ \t_R := \inf\{t\ge 0: |Q_t|\ge R\} \wedge (1-R^{-1}),
\eea
for any $R>1$. Then $\th^R$ is bounded and thus one can apply the Girsanov theorem. It is clear that $M^{\th^{R,v}}_t = M^{\th^{*v}}_{\t_R\wedge t}$. Then \reff{vM=BM} implies
$$
 \int_\dbR v M^{\th^{R,v}}_t \mu(dv)  =  B_{\t_R \wedge t} \int_\dbR  M^{\th^{R,v}}_t \mu(dv),
$$ 
and thus, by \reff{filter}, 
\bea
\label{PthRP*}
P^{\th^R}_t = B_{\t_R \wedge t} = Q_{\t_R \wedge t}= P^*_{\t_R \wedge t}.
\eea

To send $R\to \infty$, we change back to strong formulation \reff{Qth} and \reff{Pth0}. Denote 
\beaa
\hat Q_t := Q^{\th^{*V}}_t = Vt - (1-t)\int_0^t {dB_r\over 1-r},\q \hat \t_R := \t_R(\hat Q).
\eeaa
Then, for SDE \reff{Qth} with $\th'= \th^R(V,\cd)$,  $Q^{\th^R(V,\cd)}_t = \hat Q_{\hat\t_R \wedge t}$. Note that, under strong formulation \reff{PthRP*}, reads as: $P^{\th^R}(t, \hat Q_{[0,\hat \t_R \wedge t]}) = \hat Q_{\hat\t_R\wedge t}$. Thus, by \reff{Pth0}  we have
\beaa
\dbE\big[V \big|\cF^{\hat Q}_{\hat\t_R\wedge t}\big] = \hat Q_{\hat\t_R\wedge t}.
\eeaa
By \reff{thR} it is clear that $\hat\t_R\to 1$ as $R\to \infty$. Then, by sending $R\to \infty$ at above,
$\dbE\big[V \big|\cF^{\hat Q}_t\big] = \hat Q_t$, $t<1$.
That is, $P^{\th^*}(Q^{\th^{*V}}) = Q^{\th^{*V}}_t$. This exactly means $P^{\th^*} = P^*$.
\qed

\bs
\no{\bf Proof of Remark \ref{rem-cA} (iii).} Again assume $P^*(t, Q_{[0,t]}) = Q_t$, however, for each $v\in Supp(V)$, we shall consider $Q^{\th', \b}_t = \int_0^t \th'_s ds + \int_0^t \b_s dB_s$, where both $\th'$ and $\b$ are the insider's controls. Following the arguments in \cite[Section 1]{Back1992}, one may argue that in this case the reasonable expected value of the insider is:
\beaa
J_0(P^*; v, \th', \b) := \dbE\Big[\int_0^T (v- Q^{\th', \b}_t) d Q^{\th',\b}_t\Big].
\eeaa
 Similarly to \reff{bridge3},  by applying the It\^{o} formula one can easily see that
\beaa
J_0(P^*; v, \th', \b)=  \dbE\Big[v Q^{\th'}_T - {|Q^{\th'}_T|^2 - \int_0^T |\b_t|^2dt\over 2}\Big].
\eeaa
Then, unless $\b$ is required to be uniformly bounded, we have
\beaa
\sup_{\th', \b} J_0(P^*; v, \th', \b) = {v^2+1 \over 2} + {1\over 2}\sup_\b \dbE\Big[\int_0^T |\b_t|^2dt\Big] = \infty.
\eeaa
This implies that in this case $P^*$ cannot be an equilibrium. 
\qed 

\bs

\no{\bf Some discussions concerning Remark \ref{rem-Markov} (ii).} Here we argue heuristically that typically $P^*_t$ is not Markovian under $\dbP$ when $N\ge 3$. Assume for simplicity that $N=3$ and $V$ is uniform on $Supp(V) = \{-1, 0, 1\}$. Use the notation $X^v$, $v\in Supp(V)$,  in the obvious sense, then $S_t = P^*_t = X^1_t - X^{-1}_t$. By Proposition \ref{prop-MarkovianX}, we have $Z^i_t = \sum_{j=-1,0,1} \pa_{x_j} u_i(X^1_t, X^0_t, X^{-1}_t) X^{j}_s\big[\pa_z H(s, \hat Z^j_t) - \bar X_t]$.  Note that $X^1+X^0+X^{-1}=1$. Denote $\G_t := X^1_t + X^{-1}_t$. Then we have 
\beaa
X^1_t = {1\over 2}[\G_t+S_t],\q X^0_t = 1- \G_t,\q X^{-1}_t = {1\over 2}[\G_t - S_t].
\eeaa
Then we can easily derive a two dimensional SDE for $(S, \G)$:
\bea
\label{SG}
\left.\ba{lll}
\dis  S_t = \int_0^t b_1(s, S_s, \G_s) ds + \int_0^t  \si_1(s, S_s, \G_s) dB_s;\\
\dis \G_t = {2\over 3}  + \int_0^t b_2(s, S_s, \G_s) ds + \int_0^t  \si_2(s, S_s, \G_s) dB_s;
\ea\right.
\eea
where $b_1, \si_1, b_2, \si_2$ are some deterministic functions. For general $H$, typically $b_1$ and/or $\si_1$ would depend on $\G$, then $S$ would not be Markovian. For example, consider a very simple case (which is not derived from \reff{FBSDE}): $b_1 = \G-{2\over 3}$, $\si_1=1$, $b_2=0$, $\si_2=1$, then
\beaa
\G_t = {2\over 3} + B_t,\q S_t = \int_0^t B_s ds + B_t.
\eeaa
This $S$ is not Markovian under $\dbP$.
\qed

\bibliographystyle{siam}
\bibliography{refs1}

\begin{thebibliography}{10}

\bibitem{Aase2012}
{\sc K.~K. Aase, T.~Bjuland, and B.~{\O}ksendal}, {\em Strategic insider trading equilibrium: a filter theory approach}, Afrika Matematika, 23 (2012), pp.~145--162.

\bibitem{Back1992}
{\sc K.~Back}, {\em Insider trading in continuous time}, The Review of Financial Studies, 5 (1992), pp.~387--409.

\bibitem{Back2004}
{\sc K.~Back and S.~Baruch}, {\em Information in securities markets: Kyle meets \uppercase{G}losten and \uppercase{M}ilgrom}, Econometrica, 72 (2004), pp.~433--465.

\bibitem{back2021}
{\sc K.~Back, F.~Cocquemas, I.~Ekren, and A.~Lioui}, {\em Optimal transport and risk aversion in \uppercase{K}yle's model of informed trading}, Preprint, arXiv: 2006.09518,  (2020).

\bibitem{BackPedersen}
{\sc K.~Back and H.~Pedersen}, {\em Long-lived information and intraday patterns}, Journal of Financial Markets, 1 (1998), pp.~385--402.

\bibitem{BHMO}
{\sc F.~Biagini, Y.~Hu, T.~Meyer-Brandis, and B.~{\O}ksendal}, {\em Insider trading equilibrium in a market with memory}, Mathematics and Financial Economics, 6 (2012), pp.~229--247.

\bibitem{Bose2023}
{\sc S.~Bose and I.~Ekren}, {\em {\uppercase{K}yle--\uppercase{B}ack models with risk aversion and non-\uppercase{G}aussian beliefs}}, The Annals of Applied Probability, 33 (2023), pp.~4238 -- 4271.

\bibitem{Bose2024}
{\sc S.~Bose and I.~Ekren}, {\em Multidimensional \uppercase{K}yle--\uppercase{B}ack model with a risk averse informed trader}, SIAM Journal on Financial Mathematics, 15 (2024), pp.~93--120.

\bibitem{Caldentey2010}
{\sc R.~Caldentey and E.~Stacchetti}, {\em Insider trading with a random deadline}, Econometrica, 78 (2010), pp.~245--283.

\bibitem{CAMPI2011}
{\sc L.~Campi, U.~{\c C}etin, and A.~Danilova}, {\em Dynamic \uppercase{M}arkov bridges motivated by models of insider trading}, Stochastic Processes and their Applications, 121 (2011), pp.~534--567.

\bibitem{CDLL}
{\sc P.~Cardaliaguet, F.~Delarue, J.-M. Lasry, and P.-L. Lions}, {\em The Master Equation and the Convergence Problem in Mean Field Games}, AMS 201, Princeton University Press, 2019.

\bibitem{CD1}
{\sc R.~Carmona and F.~Delarue}, {\em Probabilistic Theory of Mean Field Games I -- Mean Field FBSDEs, Control, and Games}, Springer Cham, 2018.

\bibitem{Cetin2023}
{\sc U.~\c{C}etin}, {\em Insider trading with legal risk in continuous time}, Preprint, ssrn.4651481,  (2023).

\bibitem{Cetin2023b}
\leavevmode\vrule height 2pt depth -1.6pt width 23pt, {\em Insider trading with penalties, entropy and quadratic \uppercase{BSDE}s}, Preprint, arXiv: 2311.12743,  (2023).

\bibitem{Cetin2016}
{\sc U.~\c{C}etin and A.~Danilova}, {\em Markovian \uppercase{N}ash equilibrium in financial markets with asymmetric information and related forward–backward systems}, The Annals of Applied Probability, 26 (2016), pp.~1996--2029.

\bibitem{CD2}
\leavevmode\vrule height 2pt depth -1.6pt width 23pt, {\em Dynamic Markov Bridges and Market Microstructure: Theory and Applications}, Springer New York, 2018.

\bibitem{CetinXing2013}
{\sc U.~\c{C}etin and H.~Xing}, {\em {Point process bridges and weak convergence of insider trading models}}, Electronic Journal of Probability, 18 (2013), pp.~1 -- 24.

\bibitem{CCD}
{\sc J.-F. Chassagneux, D.~Crisan, and F.~Delarue}, {\em A Probabilistic Approach to Classical Solutions of the Master Equation for Large Population Equilibria}, vol.~280, Memoirs of the American Mathematical Society, 2022.

\bibitem{Cho2003}
{\sc K.-H. Cho}, {\em Continuous auctions and insider trading: uniqueness and risk aversion}, Finance and Stochastics, 7 (2003), pp.~47--71.

\bibitem{Kasper2023}
{\sc J.~H. Choi, H.~Kwon, and K.~Larsen}, {\em Trading constraints in continuous-time \uppercase{K}yle models}, SIAM Journal on Control and Optimization, 61 (2023), pp.~1494--1512.

\bibitem{CollinFos2016}
{\sc P.~Collin-Dufresne and V.~Fos}, {\em Insider trading, stochastic liquidity, and equilibrium prices}, Econometrica, 84 (2016), pp.~1441--1475.

\bibitem{Danilova2010}
{\sc A.~Danilova}, {\em Stock market insider trading in continuous time with imperfect dynamic information}, Stochastics, 82 (2010), pp.~111--131.

\bibitem{ekren2022}
{\sc I.~Ekren, B.~Mostowski, and G.~Žitković}, {\em Kyle's model with stochastic liquidity}, Finance and Stochastics, accepted,  (2025).

\bibitem{FRZ2021}
{\sc Z.~Feinstein, B.~Rudloff, and J.~Zhang}, {\em Dynamic set values for nonzero-sum games with multiple equilibriums}, Mathematics of Operations Research, 47 (2022), pp.~616--642.

\bibitem{GMMZ}
{\sc W.~Gangbo, A.~R. M{\'e}sz{\'a}ros, C.~Mou, and J.~Zhang}, {\em {Mean field games master equations with nonseparable \uppercase{H}amiltonians and displacement monotonicity}}, The Annals of Probability, 50 (2022), pp.~2178 -- 2217.

\bibitem{IseriZhang}
{\sc M.~Iseri and J.~Zhang}, {\em Set values for mean field games}, Trans. Amer. Math. Soc., 377 (2024), pp.~7117--7174.

\bibitem{KS}
{\sc G.~Kallianpur and C.~Striebel}, {\em {Estimation of Stochastic Systems: Arbitrary System Process with Additive White Noise Observation Errors}}, The Annals of Mathematical Statistics, 39 (1968), pp.~785 -- 801.

\bibitem{KS2}
{\sc G.~Kallianpur and C.~Striebel}, {\em Stochastic differential equations occurring in the estimation of continuous parameter stochastic processes}, Theory of Probability \& Its Applications, 14 (1969), pp.~567--594.

\bibitem{Kyle1985}
{\sc A.~S. Kyle}, {\em Continuous auctions and insider trading}, Econometrica, 53 (1985), pp.~1315--1335.

\bibitem{Lasserre2004}
{\sc G.~Lasserre}, {\em Asymmetric information and imperfect competition in a continuous time multivariate security model}, Finance and Stochastics, 8 (2004), pp.~285--309.

\bibitem{MSZ}
{\sc J.~Ma, R.~Sun, and Y.~Zhou}, {\em Kyle--\uppercase{B}ack equilibrium models and linear conditional mean-field \uppercase{SDE}s}, SIAM Journal on Control and Optimization, 56 (2018), pp.~1154--1180.

\bibitem{MaTan2022}
{\sc J.~Ma and Y.~Tan}, {\em A generalized \uppercase{K}yle-\uppercase{B}ack strategic insider trading model with dynamic information}, Preprint, arXiv: 2204.13137,  (2022).

\bibitem{MaXiaZhang}
{\sc J.~Ma, W.~Xia, and J.~Zhang}, {\em Wealth or stealth? the camouflage effect in insider trading}, Working Paper,  (2025).

\bibitem{MaYong}
{\sc J.~Ma and J.~Yong}, {\em Solvability of forward-backward \uppercase{SDE}s and the nodal set of \uppercase{H}amilton-\uppercase{J}acobi-\uppercase{B}ellman equations}, Chin. Ann. Math. Ser. B, 16 (1995), pp.~279--298.

\bibitem{MaZhang}
{\sc J.~Ma and J.~Zhang}, {\em On weak solutions of forward--backward \uppercase{SDE}s}, Probability Theory and Related Fields, 151 (2011), pp.~475--507.

\bibitem{MaZhangZheng}
{\sc J.~Ma, J.~Zhang, and Z.~Zheng}, {\em Weak solutions for forward-backward \uppercase{SDE}s: A martingale problem approach}, The Annals of Probability, 36 (2008), pp.~2092--2125.

\bibitem{ST}
{\sc H.~M. Soner and N.~Touzi}, {\em Dynamic programming for stochastic target problems and geometric flows}, Journal of the European Mathematical Society, 4 (2002), pp.~201--236.

\bibitem{WangZhang}
{\sc H.~Wang and J.~Zhang}, {\em Forward backward \uppercase{SDE}s in weak formulation}, Mathematical Control and Related Fields, 8 (2018), pp.~1021--1049.

\bibitem{Zhang2017}
{\sc J.~Zhang}, {\em Backward Stochastic Differential Equations: From Linear to Fully Nonlinear Theory}, Springer New York, 2017.

\end{thebibliography}

\end{document}